\documentclass[10pt]{article}

\usepackage{amsmath}
\usepackage{amssymb}
\usepackage{amsthm}
\usepackage{indentfirst}
\usepackage{enumerate}
\usepackage{cite}
\usepackage{color}
\usepackage[T1]{fontenc}
\usepackage{graphicx}
\usepackage{float}
\usepackage[makeroom]{cancel}
\usepackage{bigints}

\newtheorem{thm}{Theorem}[section]
\newtheorem{lem}{Lemma}[section]

\newtheorem{cor}{Corollary}[section]
\newtheorem{rem}{Remark}[section]

\numberwithin{equation}{section}

\title{\textbf{Associated Derived Invariants
for Geometric Mappings of Non-Symmetric Affine Connection Spaces}}

\author{Nenad O. Vesi\'c\footnote{Serbian Ministry of Education, Science and Technological
Developments through Mathematical Institute of the Serbian Academy
of Sciences and Arts}}
\date{}

\makeatletter
\def\maketag@@@#1{\hbox{\m@th\normalfont\normalsize#1}}
\makeatother

\usepackage{lipsum}

\usepackage{amsthm}

\begin{document}

  \maketitle

  \begin{abstract}
    The invariants of the
Thomas and the
    Weyl  type for a mapping between non-symmetric affine connection spaces
     are obtained with respect to the factored deformation
    tensor in this paper. Motivated by two invariants of the Weyl type obtained
    in
    \big(N. O. Vesi\'c, Basic Invariants of Geometric Mappings, \cite{jageninv}\big), we
    founded novel
    invariants of the Weyl type.
     Invariants
     for almost geodesic mappings of the third type
     are searched at the end of this paper.\\[5pt]


    \textbf{$2010$ Math. Subj. Classification:} 53A55, 53B05, 53C15
  \end{abstract}

  \section{Introduction}

  Many research papers and monographs are devoted to invariants for
  mappings between affine connection spaces. Some of them are
  J. Mike\v s \cite{mik10,mik3,mik1,mik6,mik5}, I. Hinterleitner \cite{mik6,mik5},
  N. S. Sinyukov \cite{sinjukov}, M. S. Stankovi\'c \cite{mica2, mica1, mica3}, Lj. S. Velimirovi\'c \cite{mica3},
  M. Lj. Zlatanovi\'c \cite{z4} and many others.
  The Thomas projective parameter, the Weyl conformal curvature tensor
  and the Weyl projective tensor \cite{mik1,mik6,mik5,sinjukov}
  have been studied and generalized by different mathematicians.

  Our main purpose in this paper is to obtain
  some general invariants for geometric mappings. In this
  research, we will continue the research from \cite{jageninv}.
  At the end of this paper, we will
  apply the results from this research to obtain
  invariants for the almost geodesic mappings of the third type.

  \subsection{Affine connection spaces}

  An $N$-dimensional manifold $\mathcal M_N$ equipped with the
  affine connection $\nabla$ (with torsion) is the affine connection
  space $\mathbb{GA}_N$ \big(see \cite{eisNRG, mincic4, mincic2, mincic1, z4, mica3, mica1}\big).

  The affine connection coefficients for the affine connection
  $\nabla$ are $L^i_{jk}$ and it holds $L^i_{jk}\not\equiv
  L^i_{kj}$. The symmetric and anti-symmetric part
  for the coefficients $L^i_{jk}$ are

  \begin{eqnarray}
    L^i_{\underline{jk}}=\dfrac12\big(L^i_{jk}+L^i_{kj}\big)&\mbox{and}&
    L^i_{\underset\vee{jk}}=\dfrac12\big(L^i_{jk}-L^i_{kj}\big).
    \label{eq:Lsimantisim}
  \end{eqnarray}

  The tensor $2L^i_{\underset\vee{jk}}$ is the torsion tensor for
  the space $\mathbb{GA}_N$.

  The symmetric affine connection space whose affine connection coefficients
  are $L^i_{\underline{jk}}$ is the associated space $\mathbb A_N$.

  \pagebreak

  One kind of covariant differentiation with respect to the affine connection of space
  $\mathbb A_N$ is defined. This kind of covariant derivative for a
  tensor $a^i_j$ of the type $(1,1)$ is

  \begin{equation}
    a^i_{j|k}=a^i_{j,k}+L^i_{\underline{\alpha
    k}}a^\alpha_j-L^\alpha_{\underline{jk}}a^i_\alpha,
    \label{eq:covariantderivativesim}
  \end{equation}

  \noindent for partial
  derivative $\partial/\partial x^k$ denoted by comma.

  One identity of Ricci type with respect to the
  covariant derivative (\ref{eq:covariantderivativesim})\linebreak $a^i_{j|m|n}-a^i_{j|n|m}=a^\alpha_jR_{\alpha mn}^i-
  a^i_\alpha R^\alpha_{jmn}$ is
  searched \big(see \cite{sinjukov,mik6,mik5}\big),
   for the curvature
  tensor

  \begin{equation}
    R^i_{jmn}=L^i_{\underline{jm},n}-L^i_{\underline{jn},m}+
    L^\alpha_{\underline{jm}}L^i_{\underline{\alpha n}}-
    L^\alpha_{\underline{jn}}L^i_{\underline{\alpha m}},
    \label{eq:R}
  \end{equation}

  \noindent of the space $\mathbb A_N$.

  \subsection{Recall to basic invariants}

  Let the deformation tensor for the mapping
  $\mathcal F:\mathbb{GA}_N\to\mathbb{G\overline A}_N$ be
  $P^i_{jk}=\overline\omega{}_{(p)jk}^i-\omega{}_{(p)jk}^i+
  \xi^i_{jk}$, $p,q=1,2,3$,
  $\omega{}_{(1)jk}^i=L^i_{\underline{jk}}$,
  $\omega{}_{(2)jk}^i=\omega{}^i_{jk}$,
  $\omega{}_{(3)jk}^i=-\dfrac12P^i_{\underline{jk}}$,
 the corresponding
  $\overline\omega{}_{(p)jk}^i$, and
  $\xi^i_{jk}=-\xi^i_{kj}$ as well.

  The basic associated invariants of the Thomas type of the first, the second
  and the third type for the mapping $\mathcal F$
  are \cite{jageninv}

  \begin{eqnarray}
    \widetilde{\mathcal
    T}{}^i_{(1)jk}=0,&
    \widetilde{\mathcal T}{}^i_{jk}=
    \widetilde{\mathcal
    T}{}^i_{(2)jk}=L^i_{\underline{jk}}-\omega_{jk}^i,&
    \widetilde{\mathcal
    T}{}^i_{(3)jk}=\dfrac12\big(\overline L{}^i_{\underline{jk}}+L^i_{\underline{jk}}\big).
    \label{eq:basicThomas}
  \end{eqnarray}

  With respect to the transformation rule of the curvature tensor
  $R^i_{jmn}$ of the associated space $\mathbb{A}_N$, it is obtained
  the basic associated invariant of the Weyl type \cite{jageninv} for the mapping $\mathcal F$

  \begin{equation}
    {\widetilde{\mathcal W}}{}^i_{jmn}=
    \widetilde{\mathcal W}{}^i_{(2)jmn}=
    R^i_{jmn}-\omega{}^i_{jm|n}+
    \omega{}^i_{jn|m}+
    \omega{}^\alpha_{jm}\omega{}^i_{\alpha
    n}-\omega{}^\alpha_{jn}\omega{}^i_{\alpha
    m}.
    \label{eq:Wbasic}
  \end{equation}

  In this paper, we are interested to obtain associated invariants for the mapping
  $\mathcal F$ with respect to the factored geometrical object
  $\omega^i_{jk}$, i.e. for
  $\omega^i_{jk}=s_1(\delta^i_j\rho_k+\delta^i_k\rho_j)+s_2(
  f^i_j\sigma_k+f^i_k\sigma_j)+s_3\phi^i_{jk}$,
  $s_1,s_2,s_3\in\{0,1\}$.

  \subsection{Motivation}

  In \cite{jageninv}, it is obtained associated invariants for
  mappings whose deformation tensors are
  $P^i_{{jk}}=\overline\omega{}^i_{jk}-\omega^i_{jk}+\xi^i_{jk}$.
  To test the efficiency of the obtained results, the author
  searched the invariants for a mapping $\mathcal
  F:\mathbb{GA}_N\to\mathbb{G\overline A}{}_N$ with respect to
  \big(see Corollary 1 in \cite{jageninv}\big)

  \begin{equation}
  \omega^i_{jk}=\delta^i_ju_k+\delta^i_ku_j+\sigma^i_{jk}
  \label{eq:omegabasicinvariants}
  \end{equation}

  The basic associated invariants for a geodesic mapping $\mathcal
  F:\mathbb{G
  A}_N\to\mathbb{G\overline A}{}_N$ are\linebreak \big(because $\omega^i_{jk}=\dfrac1{N+1}\delta^i_jL^\alpha_{\underline{k\alpha}}
  +\delta^i_k\dfrac1{N+1}L^\alpha_{\underline{j\alpha}}$\big)

  \begin{align}
    &\aligned
    \widetilde{\mathcal
    T}{}^i_{jk}=L^i_{\underline{jk}}-\dfrac1{N+1}\big(\delta^i_jL^\alpha_{\underline{k\alpha}}+
    \delta^i_kL^\alpha_{\underline{k\alpha}}\big),
    \endaligned\label{eq:Thomasprojp}\\
    &\aligned
    \widetilde{\mathcal W}{}^i_{jmn}&=R^i_{jmn}+\dfrac1{N+1}\delta^i_jL^\alpha_{[\underline{m\alpha}|n]}-
    \dfrac1{(N+1)^2}\delta^i_m\big((N+1)L^\alpha_{\underline{j\alpha}|n}+
    L^\alpha_{\underline{j\alpha}}L^\beta_{\underline{n\beta}}\big)\\&+
    \dfrac1{(N+1)^2}\delta^i_n\big((N+1)L^\alpha_{\underline{j\alpha}|m}+
    L^\alpha_{\underline{j\alpha}}L^\beta_{\underline{m\beta}}\big)
    \endaligned\label{eq:Weylprojtstart}
  \end{align}

  \noindent for
  $L^i_{\underline{jm}|n}=L^i_{\underline{jm},n}+L^i_{\underline{\alpha
  n}}L^\alpha_{\underline{jm}}-L^\alpha_{\underline{jn}}L^i_{\underline{\alpha
  m}}+L^\alpha_{\underline{mn}}L^i_{\underline{j\alpha}}$.

  The basic associated invariant (\ref{eq:Weylprojtstart}) is different of the
  Weyl projective tensor\linebreak
  $W^i_{jmn}=R^i_{jmn}+\dfrac1{N+1}\delta^i_jR_{[mn]}+
  \dfrac N{N^2-1}\delta^i_{[m}R_{jn]}+
  \dfrac1{N^2-1}\delta^i_{[m}R_{n]j}$.

  This paper is consisted of the introduction, three sections and
  conclusion.

  \begin{enumerate}
    \item At the start of the research, we will present the iterative
    rule for obtaining novel associated invariants of the Weyl type for a
    mapping $\mathcal F:\mathbb{GA}_N\to\mathbb{G\overline A}{}_N$
    with respect to the known ones.
    \item In the next section of this papper, we will pay attention to the basic
    invariants of the Thomas and the Weyl types (the invariants $\widetilde{\mathcal
    T}{}^i_{jk}$
    and $\widetilde{\mathcal W}{}^i_{jmn}$) for a mapping $\mathcal
    F:\mathbb{GA}_N\to\mathbb{G\overline A}{}_N$. Motivated with the
    results presented in the
    Corollary 2.2 in \cite{jageninv}, we will obtain associated invariants of the Weyl type
    with respect to the equality $\widetilde{\overline{\mathcal
    W}}{}^i_{jmn}-\widetilde{\mathcal{W}}{}^i_{jmn}=0$ and
    generalize them by applying the iterative process from the first
    section of this paper.
    \item In the fourth section of this paper, we will apply the
    obtained results to find associated invariants of the Thomas and
    the Weyl type for an almost geodesic mapping of the third type.
  \end{enumerate}

  \section{Associated derived
  invariants}\label{associatedderivedinvariantssec2}

  With
    respect to the equations (\ref{eq:covariantderivativesim},
    \ref{eq:R}), we get

  \begin{equation}
  R^\alpha_{\alpha ij}=L^\alpha_{\underline{\alpha i},j}-
      L^\alpha_{\underline{\alpha j},i}=
      L^\alpha_{\underline{\alpha i}|j}-
      L^\alpha_{\underline{\alpha
      j}|i}=-R_{[ij]}.\label{eq:Wbasicfactored**}
      \end{equation}

  Let us prove the following theorem.

  \begin{thm}\label{theoreminv}
    Any invariant $\mathcal W^i_{jmn}$ for a geometrical mapping $\mathcal
    F:\mathbb{GA}_N\to\mathbb{G\overline A}{}_N$ obtained with
    respect to the transformation rule of the curvature tensor
    $R^i_{jmn}$ of the associated space $\mathbb{A}_N$ is
    anti-symmetric by the indices $m$ and $n$.

    If the geometrical object

    \begin{equation}
      W^i_{jmn}=R^i_{jmn}+\delta^i_jX_{[mn]}+\delta^i_{[m}Y_{jn]}+Z^i_{jmn},
      \label{eq:propinvW}
    \end{equation}

    \noindent for tensors $X_{ij}$ and
    $Y_{ij}$ of the type $(0,2)$, and a tensor $Z^i_{jmn}$, $Z^i_{jmn}=-Z^i_{jnm}$, of the
    type $(1,3)$, is an invariant for a mapping $\mathcal F:\mathbb{GA}_N\to\mathbb{G\overline A}{}_N$,
    then the geometrical objects

    \begin{align}
      &\aligned
      W_{jmn}^{(1)i}&=R^i_{jmn}-\dfrac1N\delta^i_j\big(Y_{[mn]}+Z^\alpha_{\alpha
      mn}\big)+\delta^i_{[m}Y_{jn]}+Z^i_{jmn},
      \endaligned\label{eq:condinv1}\\\displaybreak[0]
      &\aligned
      W^{(2)i}_{jmn}&=R^i_{jmn}-\dfrac12\delta^i_j\big((N-1)Y_{[mn]}-Z^\alpha_{[mn]\alpha}\big)+
      \delta^i_{[m}Y_{jn]}+Z^i_{jmn},
      \endaligned\label{eq:condinv2}\\
      &\aligned
      W^{(4)i}_{jmn}&=R^i_{jmn}+\dfrac1{N-1}\delta^i_{[m}R_{\underline{jn}]}+
      \delta^i_jX_{[mn]}+Z^i_{jmn}\\&-\frac1{N-1}\big(\delta^i_{[m}X_{jn]}-\delta^i_{[m}X_{n]j}\big)+
      \dfrac1{N-1}\delta^i_{[m}Z^\alpha_{jn]\alpha},
      \endaligned\label{eq:condinv4}\\
      &R_{[ij]}=R_{ij}-R_{ji},\label{eq:R[ij]inv}
      \end{align}

    \noindent for $R_{\underline{ij}}=\dfrac12\big(R_{ij}+R_{ji}\big)$, are invariants for this mapping.
  \end{thm}

  \begin{proof}
    Let us consider the transformation rule

    \begin{equation}
      \overline
      R{}^i_{jmn}=R^i_{jmn}+\overline\Pi{}^i_{jmn}-\Pi_{jmn}^i,
      \label{eq:RtoR=Pi-Pi}
    \end{equation}

    \noindent for the corresponding geometrical object
    $\Pi{}^i_{jmn}$ and its image $\overline\Pi{}^i_{jmn}$.

    \pagebreak

    Let also be

    \begin{equation}
      \begin{array}{cc}
        \Pi^i_{j(mn)}=\Pi^i_{jmn}+\Pi^i_{jnm},&
        \Pi^i_{j[mn]}=\Pi^i_{jmn}-\Pi^i_{jnm},\\
        \overline\Pi{}^i_{j(mn)}=\overline\Pi{}^i_{jmn}+\overline\Pi{}^i_{jnm},&
        \overline\Pi{}^i_{j[mn]}=\overline\Pi{}^i_{jmn}-\overline\Pi{}^i_{jnm}.
      \end{array}\label{eq:Pis}
    \end{equation}

    With respect to the equation (\ref{eq:R}), one reads

    \begin{equation}
    \begin{array}{ccc}
    R^i_{jmn}=-R^i_{jnm}&\mbox{which means}&R^i_{jmn}=\dfrac12\big(R^i_{jmn}-R^i_{jnm}\big).
    \end{array}\label{eq:Rijmn=-Rijnm}
    \end{equation}

    Based on the equations (\ref{eq:RtoR=Pi-Pi},
    \ref{eq:Pis}, \ref{eq:Rijmn=-Rijnm}), we get

    \begin{equation}
      \overline R{}^i_{jmn}-R^i_{jmn}=
      \dfrac12\big(\overline\Pi{}^i_{j[mn]}-\Pi^i_{j[mn]}\big)+
      \dfrac12\big(\overline\Pi{}^i_{j(mn)}-\Pi^i_{j(mn)}\big).
      \label{eq:R-Rantisimetric1}
    \end{equation}

    After symmetrizing the last equation by the indices $m$ and $n$,
    one obtains

    \begin{equation}
      0=
      \overline\Pi{}^i_{j(mn)}-\Pi^i_{j(mn)}.
      \label{eq:R-Rantisimetric2}
    \end{equation}

    If substitutes the equation (\ref{eq:R-Rantisimetric2}) into the
    equation (\ref{eq:R-Rantisimetric1}), one will conclude that the
    geometrical object

    \begin{equation*}
      \dfrac12\big(R^i_{jmn}-R^i_{jnm}-\Pi^i_{j[mn]}\big)
    \end{equation*}

    \noindent is an invariant for this mapping. This invariant is
    anti-symmetric by the indices $m$ and $n$, which
    completes the proof of the first part for this theorem.

    Let us consider the equality $0=\overline
    W{}^i_{jmn}-W^i_{jmn}$, i.e.

    \begin{equation}
      \aligned
      0&=\overline R{}^i_{jmn}\!-\!R^i_{jmn}\!+\!\delta^i_j\big(\overline
      X{}_{[mn]}\!-\!X_{[mn]}\big)\!+\!\big(\delta^i_{[m}\overline
      Y{}_{jn]}\!-\!\delta^i_{[m}Y_{jn]}\big)\!+\!\big(\overline
      Z{}^i_{jmn}\!-\!Z^i_{jmn}\big).
      \endaligned\label{eq:propositionproof1}
    \end{equation}

    After contracting it by the indices $i$ and $j$ and applying
    the equation (\ref{eq:Wbasicfactored**}), one
    gets

    \begin{equation}
      0=-\overline R_{[mn]}+R_{[mn]}+N\big(\overline
      X{}_{[mn]}-X_{[mn]}\big)+\big(\overline
      Y_{[mn]}-Y_{[mn]}\big)+\big(\overline Z{}^\alpha_{\alpha
      mn}-Z^\alpha_{\alpha mn}\big).\label{eq:propositionproof1i=j}
    \end{equation}

    If expresses the summand $\overline X{}_{[mn]}-X_{[mn]}$
    with respect to the equation (\ref{eq:propositionproof1i=j}) and
    substitutes it into the equality (\ref{eq:propositionproof1}),
    one will confirm the invariance $\overline{\mathcal
    W}{}^{(1)i}_{jmn}=\mathcal W^{(1)i}_{jmn}$, for

    \begin{equation}
      \mathcal W^{(1)i}_{jmn}=R^i_{jmn}+\dfrac1N\delta^i_j\big(R_{[mn]}-Y_{[mn]}-Z^\alpha_{\alpha
      mn}\big)+\delta^i_{[m}Y_{jn]}+Z^i_{jmn},
      \label{eq:proofconditionalcalW1}
    \end{equation}

    \noindent and the corresponding $\mathcal{\overline
    W}{}^{(1)i}_{jmn}$.

    By contracting the equation (\ref{eq:propositionproof1}) on
    the indices $i$ and $n$, we obtain

    \begin{align}
      &\aligned
      0=\overline R{}_{jm}-R_{jm}-\big(\overline
      X_{[jm]}-X_{[jm]}\big)-(N-1)\big(\overline
      Y{}_{jm}-Y_{jm}\big)+
      \big(\overline
      Z{}^\alpha_{jm\alpha}-Z^\alpha_{jm\alpha}\big),
      \endaligned\label{eq:propositionproof1i=n}\\
      &\aligned
      0=\overline R{}_{[jm]}\!-\!R_{[jm]}\!-\!2\big(\overline
      X_{[jm]}\!-\!X_{[jm]}\big)\!-\!(N-1)\big(\overline
      Y{}_{[jm]}\!-\!Y_{[jm]}\big)\!+\!
      \big(\overline
      Z{}^\alpha_{[jm]\alpha}\!-\!Z^\alpha_{[jm]\alpha}\big).
      \endaligned\label{eq:propositionproof1i=nn}
    \end{align}

    Based on the equations (\ref{eq:propositionproof1i=n},
    \ref{eq:propositionproof1i=nn}), we also get

    \begin{align}
      &\aligned
      \overline X{}_{[ij]}-X_{[ij]}=\overline
      R{}_{ij}-(N-1)\overline Y_{ij}+\overline
      Z{}^\alpha_{ij\alpha}-R_{ij}+(N-1)Y_{ij}-Z^\alpha_{ij\alpha},
      \endaligned\label{eq:propositionproof1X-X1}\\\displaybreak[0]
      &\aligned
      \overline
      X{}_{[ij]}-X_{[ij]}&=\dfrac12\big(\overline
      R{}_{[ij]}-(N-1)\overline Y{}_{[ij]}+\overline
      Z{}^\alpha_{[ij]\alpha}\big)-\dfrac12\big(R_{[ij]}-(N-1)Y_{[ij]}+Z^\alpha_{[ij]\alpha}\big),
      \endaligned\label{eq:propositionproof1X-X2}\\\displaybreak[0]
      &\aligned
      \overline
      Y{}_{ij}-Y_{ij}&=\dfrac1{N-1}\big(\overline
      R{}_{ij}-\overline X{}_{[ij]}+\overline
      Z{}^\alpha_{ij\alpha}\big)-\dfrac1{N-1}\big(R_{ij}-X_{[ij]}+Z^\alpha_{ij\alpha}\big).
      \endaligned\label{eq:propositionproof1Y-Y1}
    \end{align}

    \pagebreak

    If substitute the expressions
    (\ref{eq:propositionproof1X-X1}, \ref{eq:propositionproof1X-X2},
    \ref{eq:propositionproof1Y-Y1})
    into the equation (\ref{eq:propositionproof1}) and use the equality\linebreak $\delta^i_mX_{[jn]}-\delta^i_nX_{[jm]}=
    \delta^i_{[m}X_{jn]}-\delta^i_{[m}X_{n]j}$, we will obtain

    \begin{equation*}
      \begin{array}{ccc}
        \mathcal{\overline W}{}^{(2)i}_{jmn}=\mathcal W^{(2)i}_{jmn},&
        \mathcal{\overline W}{}^{(3)i}_{jmn}=\mathcal W^{(3)i}_{jmn},&
        \mathcal{\overline W}{}^{(4)i}_{jmn}=\mathcal W^{(4)i}_{jmn},
      \end{array}
    \end{equation*}

    \noindent for

    \begin{align}
      &\aligned\mathcal W^{(2)i}_{jmn}&=R^i_{jmn}+\delta^i_j\big(R_{[mn]}-(N-1)Y_{mn}+Z^\alpha_{mn\alpha}\big)+
      \delta^i_{[m}Y_{jn]}+Z^i_{jmn},\endaligned\label{eq:proofconditionalcalW2}
      \\
      &\aligned\mathcal W^{(3)i}_{jmn}&=R^i_{jmn}+\dfrac12\delta^i_j\big(R_{[mn]}-(N-1)Y_{[mn]}+Z^\alpha_{[mn]\alpha}\big)+
      \delta^i_{[m}Y_{jn]}+Z^i_{jmn},
      \endaligned\label{eq:proofconditionalcalW3}\\
      &\aligned
      \mathcal W^{(4)i}_{jmn}&=R^i_{jmn}+\dfrac1{N-1}\delta^i_{[m}R_{\underline{jn}]}+
      \delta^i_jX_{[mn]}+Z^i_{jmn}
      \\&-\frac1{N-1}\big(\delta^i_{[m}X_{jn]}-\delta^i_{[m}X_{n]j}\big)+
      \dfrac1{N-1}\delta^i_{[m}Z^\alpha_{jn]\alpha},
      \endaligned\label{eq:proofconditionalcalW4}
    \end{align}

    \noindent and the corresponding $\overline{\mathcal
    W}{}^{(2)i}_{jmn}$, $\overline{\mathcal
    W}{}^{(3)i}_{jmn}$, $\overline{\mathcal
    W}{}^{(4)i}_{jmn}$.

    With respect to the equation (\ref{eq:proofconditionalcalW2})
    and the first part of this theorem \big($\mathcal
    W^{(2)i}_{jmn}=-\mathcal W^{(2)i}_{jnm}$\big), the
    invariant $\mathcal W^{(2)i}_{jmn}$ satisfies the equality $\mathcal
    W^{(2)i}_{jmn}=\dfrac12\big(\mathcal W^{(2)i}_{jmn}-\mathcal
    W^{(2)i}_{jnm}\big)$, i.e.

    \begin{equation}
      \mathcal W^{(2)i}_{jmn}=R^i_{jmn}+\dfrac12\delta^i_j
      \big(2R_{[mn]}-(N-1)Y_{[mn]}+Z^\alpha_{[mn]\alpha}\big)+
      \delta^i_{[m}Y_{jn]}+Z^i_{jmn}.
      \tag{\ref{eq:proofconditionalcalW2}'}\label{eq:proofconditionalcalW2'}
    \end{equation}

    From the equations (\ref{eq:proofconditionalcalW3}) and
    (\ref{eq:proofconditionalcalW2'}), we find the invariant

    \begin{equation}
      \mathcal W{}^{(2)i}_{jmn}-\mathcal W{}^{(3)i}_{jmn}=
      \dfrac12\delta^i_jR_{[mn]},
      \tag{\ref{eq:R[ij]inv}'}\label{eq:R[ij]inv'}
    \end{equation}

    \noindent for the mapping $\mathcal F$, which proves the stated invariance of the
    geometrical object $R_{[ij]}$.

    For this reason, the invariants for the mapping $\mathcal F$ obtained in the
    equations (\ref{eq:proofconditionalcalW1},
    \ref{eq:proofconditionalcalW2}, \ref{eq:proofconditionalcalW3},
    \ref{eq:proofconditionalcalW4}) reduce to the geometrical
    objects $W^{(1)i}_{jmn}$, $W^{(2)i}_{jmn}$,
    $W^{(4)i}_{jmn}$ given by the equations (\ref{eq:condinv1},
    \ref{eq:condinv2}, \ref{eq:condinv4}).
    \end{proof}

  The invariants $W^{(1)i}_{jmn}$, $W^{(2)i}_{jmn}$ and $W^{(4)i}_{jmn}$ for the mapping
  $\mathcal F:\mathbb{GA}_N\to\mathbb{G\overline A}{}_N$ given by the equations
   (\ref{eq:condinv1}, \ref{eq:condinv2}, \ref{eq:condinv4}) are
  \emph{the first}, \emph{the second} and \emph{the fourth associated derived invariant} \big(with
  respect to the invariant $W^i_{jmn}$ given by the equation
  (\ref{eq:propinvW})\big).

  \begin{rem}
    If $X_{[ij]}=0$, the equations
    \emph{(\ref{eq:propositionproof1X-X1},
    \ref{eq:propositionproof1X-X2})}
    reduce to

    \begin{align}
      &\aligned
     0=\overline
      R{}_{ij}-(N-1)\overline Y_{ij}+\overline
      Z{}^\alpha_{ij\alpha}-R_{ij}+(N-1)Y_{ij}-Z^\alpha_{ij\alpha},
      \endaligned\label{eq:propositionproof1X-X1'}\\\displaybreak[0]
      &\aligned
      0&=\dfrac12\big(\overline
      R{}_{[ij]}-(N-1)\overline Y{}_{[ij]}+\overline
      Z{}^\alpha_{[ij]\alpha}\big)-\dfrac12\big(R_{[ij]}-(N-1)Y_{[ij]}+Z^\alpha_{[ij]\alpha}\big).
      \endaligned\label{eq:propositionproof1X-X2'}
    \end{align}

    By anti-symmetrizing the equation
    \emph{(\ref{eq:propositionproof1X-X1'})} on the indices $i$ and
    $j$, one will obtain

    \begin{equation}
      0=\overline
      R{}_{[ij]}-(N-1)\overline Y_{[ij]}+\overline
      Z{}^\alpha_{[ij]\alpha}-R_{[ij]}+(N-1)Y_{[ij]}-Z^\alpha_{[ij]\alpha}.
      \tag{\ref{eq:propositionproof1X-X1'}'}\label{eq:propositionproof1X-X1''}
    \end{equation}

    From the difference
    $2\cdot(\ref{eq:propositionproof1X-X2'})-(\ref{eq:propositionproof1X-X1''})$,
    one derives the invariants $\delta^i_jR_{[mn]}$ and $\dfrac1N\delta^\alpha_\alpha
    R_{[mn]}\equiv R_{[mn]}$ for the mapping $\mathcal F$.
  \end{rem}
  \section{Associated invariants for geometric
  mappings}\label{associatedderivedinvariantssec3}

  Let $\mathcal F:\mathbb{GA}_N\to\mathbb{G\overline A}_N$ be a mapping which
  transforms the affine connection coefficients $L^i_{jk}=L^i_{\underline{jk}}+
  L^i_{\underset\vee{jk}}$ of the space $\mathbb{GA}_N$ by the rule

  \begin{equation}
    \aligned
    \overline L^i_{{jk}}&=L^i_{{jk}}+s_1\Big[\delta^i_j\big(\overline
    u_k-u_k\big)+\delta^i_k\big(\overline u_j-u_j\big)\Big]\\&+
    s_2\Big[\big(
    \overline f{}^i_j\overline\sigma{}_k+\overline
    f{}^i_k\overline\sigma_j\big)-\big(f^i_j\sigma_k+f^i_k\sigma_j\big)\Big]+
    s_3\Big[\overline\phi{}^i_{jk}-\phi^i_{jk}\Big]+\xi^i_{jk},
    \endaligned\label{eq:LtoLtransformationrulegeneral}
  \end{equation}

  \noindent for the corresponding coefficients
  $s_1,s_2,s_3\in\{0,1\}$, the tensor $\xi^i_{jk}$ of the type $(1,2)$ anti-symmetric by the indices $j$ and $k$, the covariant vectors
  $u_j$, $\overline u{}_j$, $\sigma_j$, $\overline\sigma{}_j$,
  the contravariant vectors $\phi^i$,
  $\overline\phi{}^i$, the affinors $f^i_j$, $\overline f{}^i_j$, and the
  geometrical objects $\phi^i_{jk}$ and $\overline\phi{}^i_{jk}$ of the type $(1,2)$ such that $\phi^i_{jk}=\phi^i_{kj}$ and
  $\overline\phi{}^i_{jk}=\overline\phi{}^i_{kj}$.

  After symmetrizing the last equation by the
  indices $j$ and $k$, one gets

  \begin{align}
    &\aligned
    \overline L^i_{\underline{jk}}&=L^i_{\underline{jk}}+s_1\Big[\delta^i_j\big(\overline
    u_k-u_k\big)+\delta^i_k\big(\overline u_j-u_j\big)\Big]\\&+
    s_2\Big[\big(
    \overline f{}^i_j\overline\sigma{}_k+\overline
    f{}^i_k\overline\sigma_j\big)-\big(f^i_j\sigma_k+f^i_k\sigma_j\big)\Big]+
    s_3\Big[\overline\phi{}^i_{jk}-\phi^i_{jk}\Big].
    \endaligned
    \label{eq:LtoLtransformationrulegeneralsim}
    \end{align}

    \subsection{Basic associated invariant of Thomas type}

  If contract the equality
  (\ref{eq:LtoLtransformationrulegeneralsim}) by the indices
  $i$ and $j$ and use the equalities $\phi^i_{jk}=\phi^i_{kj}$ and
  $\overline\phi{}^i_{jk}=\overline\phi{}^i_{kj}$, we will obtain

  \begin{equation*}
    s_1\psi_k=\frac 1{N+1}\Bigg\{\overline L{}^\alpha_{\underline{k\alpha}}-
    L^\alpha_{\underline{k\alpha}}-s_2\Big(\big(\overline f{}^\alpha_k\overline\sigma_\alpha+
    \overline f\overline\sigma{}_k\big)-
    \big(f^\alpha_k\sigma_\alpha+f\sigma_k\big)\Big)-
    s_3\big(\overline\phi{}^\alpha_{k\alpha}-
    \phi^\alpha_{k\alpha}\big)\Bigg\},
  \end{equation*}

  \noindent for $\psi_i=\overline u{}_i-u_i$,
  the scalars $f=f^\alpha_\alpha$ and $\overline f=\overline f{}^\alpha_\alpha$.

  After substituting the previous expression of $s_1\psi_j$ into   the equation
  (\ref{eq:LtoLtransformationrulegeneralsim}), one
  obtains

  \begin{align}
    &\aligned
    \omega^i_{jk}&=s_2\big(f^i_j\sigma_k+f^i_k\sigma_j\big)+s_3\phi^i_{jk}
    +\frac1{N+1}\delta^i_j\Big[L^\alpha_{\underline{k\alpha}}-
    s_2\big(f^\alpha_k\sigma_\alpha+f\sigma_k\big)-s_3\phi^\alpha_{k\alpha}\Big]
    \\&+\frac1{N+1}\delta^i_k\Big[L^\alpha_{\underline{j\alpha}}-
    s_2\big(f^\alpha_j\sigma_\alpha+f\sigma_j\big)-s_3\phi^\alpha_{j\alpha}\Big]
    .
    \endaligned\label{eq:omegageneral}
  \end{align}

  The basic associated invariant of the Thomas type for the mapping $\mathcal F$
  is \cite{jageninv}

  \begin{align}
    &\aligned
    \widetilde{\mathcal T}{}^i_{jk}&=L^i_{\underline{jk}}-
    s_2\big(f^i_j\sigma_k+f^i_k\sigma_j\big)-s_3\phi^i_{jk}-
    \dfrac1{N+1}\Big\{\delta^i_j\Big[
    L^\alpha_{\underline{k\alpha}}-s_2\big(f^\alpha_k\sigma_\alpha+f\sigma_k\big)
    -s_3\phi^\alpha_{k\alpha}
    \Big]\\&+
    \delta^i_k\Big[
    L^\alpha_{\underline{j\alpha}}-
    s_2\big(f^\alpha_j\sigma_\alpha+f\sigma_j\big)-
    s_3\phi^\alpha_{j\alpha}
    \Big]
    \Big\}.
    \endaligned\label{eq:ThomasBasicsim2factoredsim}
  \end{align}

  It holds the next lemma.
  \begin{lem}\label{thomasbasicfactored}
    Let $f:\mathbb{GA}_N\to\mathbb{G\overline A}_N$ be a mapping
    defined on a non-symmetric affine connection space
    $\mathbb{GA}_N$. The geometrical object
    \emph{(\ref{eq:ThomasBasicsim2factoredsim})} is the basic associated invariant of the Thomas
    type for this mapping.\qed
  \end{lem}
  \begin{cor}
    Let $\mathcal F:\mathbb{GA}_N\to\mathbb{G\overline A}{}_N$ be a
    mapping of a non-symmetric affine connection space
    $\mathbb{GA}_N$.

    In the case of $s_1=0$, the geometrical object

    \begin{equation}
      \tilde\theta_i=L^\alpha_{\underline{i\alpha}}-s_2\big(f^\alpha_i\sigma_\alpha+f\sigma_i\big)-s_3\phi^\alpha_{i\alpha}
      \label{eq:thetainvantisim}
    \end{equation}

    \noindent is an invariant for the mapping $\mathcal F$.

    In this
    case, the invariant \emph{(\ref{eq:ThomasBasicsim2factoredsim})}
    reduces to

    \begin{equation}
      \overset\ast{\widetilde{\mathcal
      T}}{}^i_{jk}=L^i_{\underline{jk}}-s_2\big(f^i_j\sigma_k+f^i_k\sigma_j\big)-s_3\phi^i_{jk},
      \label{eq:basicthomass1=0}
    \end{equation}

  \noindent for the corresponding coefficients $s_2$, $s_3$.\qed
  \end{cor}

  \subsection{Basic associated invariant of Weyl type}

  For the geometrical object $\omega{}^i_{jk}$ given by the
  equation (\ref{eq:omegageneral}), one gets

  \begin{footnotesize}
  \begin{align}
    &\aligned
    \omega^i_{jm|n}&=s_2\big(f^i_{j|n}\sigma_m+f^i_{m|n}\sigma_j+f^i_j\sigma_{m|n}+
    f^i_m\sigma_{j|n}\big)+s_3\phi^i_{jm|n}\\&+
    \frac1{N+1}\delta^i_j\Big[
    L^\alpha_{\underline{m\alpha}|n}-s_2\big(f^\alpha_{m|n}\sigma_\alpha+f_{,n}\sigma_m
    +f^\alpha_m\sigma_{\alpha|n}+f\sigma_{m|n}
    \big)-s_3\phi^\alpha_{m\alpha|n}
    \Big]\\&+
    \frac1{N+1}\delta^i_m\Big[
    L^\alpha_{\underline{j\alpha}|n}-s_2\big(
    f^\alpha_{j|n}\sigma_\alpha+f_{,n}\sigma_j+
    f^\alpha_j\sigma_{\alpha|n}+f\sigma_{j|n}
    \big)-
    s_3
    \phi^\alpha_{j\alpha|n}
    \Big],
    \endaligned\label{eq:omegaijm|nfactored}\\\displaybreak[0]
    &\aligned
    \omega^\alpha_{jm}\omega^i_{\alpha n}&=
    \dfrac2{(N+1)^2}\delta^i_n\Big[L^\alpha_{\underline{j\alpha}}
    -s_2\big(f^\alpha_j\sigma_\alpha+f\sigma_j\big)-s_3\phi^\alpha_{j\alpha}
    \Big]\Big[L^\beta_{\underline{m\beta}}-s_2\big(f^\beta_m\sigma_\beta+f\sigma_m\big)-
    s_3\phi^\beta_{m\beta}
    \Big]\\&+
    \dfrac1{(N+1)^2}\delta^i_m\Big[L^\alpha_{\underline{j\alpha}}
    -s_2\big(f^\alpha_j\sigma_\alpha+f\sigma_j\big)-s_3\phi^\alpha_{j\alpha}
    \Big]\Big[L^\beta_{\underline{n\beta}}-s_2\big(f^\beta_n\sigma_\beta+f\sigma_n\big)-
    s_3\phi^\beta_{n\beta}
    \Big]\\&+
    \dfrac1{(N+1)^2}\delta^i_j
    \Big[L^\alpha_{\underline{m\alpha}}-
    s_2\big(f^\alpha_m\sigma_\alpha+f\sigma_m\big)-
    s_3\phi^\alpha_{m\alpha}\Big]
    \Big[
    L^\beta_{\underline{n\beta}}-
    s_2\big(f^\beta_n\sigma_\beta+f\sigma_n\big)-
    s_3\phi^\beta_{n\beta}
    \Big]\\&+
    \frac1{N+1}\delta^i_n
    \Big[s_2\big(
    f^\alpha_j\sigma_m+f^\alpha_m\sigma_j
    \big)+s_3\sigma_{jm}\phi^\alpha
    \Big]\Big[
    L^\beta_{\underline{\alpha\beta}}-
    s_2\big(f^\beta_\alpha\sigma_\beta+f\sigma_\alpha\big)-
    s_3\sigma_{\alpha\beta}\phi^\beta
    \Big]\\&+\dfrac1{N+1}
    \Big[L^\alpha_{\underline{n\alpha}}-s_2
    \big(f^\alpha_n\sigma_\alpha+f\sigma_n\big)-
    s_3\phi^\alpha_{n\alpha}\Big]\Big[s_2
    \big(f^i_j\sigma_m+f^i_m\sigma_j\big)+
    s_3\phi^i_{jm}\Big]\\&+
    \dfrac1{N+1}\Big[L^\alpha_{\underline{m\alpha}}-
    s_2\big(f^\alpha_m\sigma_\alpha+f\sigma_m\big)-
    s_3\phi^\alpha_{m\alpha}\Big]\Big[
    s_2\big(f^i_j\sigma_n+f^i_n\sigma_j\big)+s_3\phi^i_{jn}\Big]\\&+
    \dfrac1{N+1}\Big[L^\alpha_{\underline{j\alpha}}-
    s_2\big(f^\alpha_j\sigma_\alpha+f\sigma_j\big)-s_3\phi^\alpha_{j\alpha}\Big]
    \Big[s_2\big(f^i_m\sigma_n+f^i_n\sigma_m\big)+s_3\phi^i_{mn}\Big]\\&+
    \Big[s_2\big(f^\alpha_j\sigma_m+f^\alpha_m\sigma_j\big)+s_3\phi^\alpha_{jm}\Big]
    \Big[s_2\big(f^i_\alpha\sigma_n+f^i_n\sigma_\alpha\big)+s_3\phi^i_{\alpha n}\Big].
    \endaligned\label{eq:omegaomegafactored}
  \end{align}
  \end{footnotesize}

  If substitute the expressions (\ref{eq:omegaijm|nfactored},
  \ref{eq:omegaomegafactored}) into the equation (\ref{eq:Wbasic}),
  and use the invariance $\overline R{}_{[ij]}=R_{[ij]}$ as well,
  we will obtain that the geometrical object

    \begin{equation}
      \widetilde{\mathcal
      W}{}^i_{jmn}=R^i_{jmn}+\mathcal A^i_{jmn}-\dfrac1{N+1}\Big[\delta^i_j\rho_{[mn]}
      +\delta^i_{[m}L^\alpha_{\underline{j\alpha}|n]}-
      \delta^i_{[m}\rho_{jn]}\Big]
      -\dfrac1{(N+1)^2}\delta^i_{[m}\widetilde{\mathcal
      S}{}_{jn]},
      \label{eq:Wbasicfactored}
    \end{equation}

    \noindent
    for $f_i=f_{,i}\equiv f_{|i}$ and

    \begin{align}
      &\rho_{ij}=s_2\big(f^\alpha_{i|j}\sigma_\alpha+f_{j}\sigma_i+
      f^\alpha_i\sigma_{\alpha|j}+f\sigma_{i|j}\big)+
      s_3\phi^\alpha_{i\alpha|j},
      \label{eq:Wbasicfactoredrho}\\\displaybreak[0]
      &\aligned
      \widetilde{\mathcal
      S}{}_{ij}&=(N+1)\Big[
      L^\beta_{\underline{\alpha\beta}}-
      s_2\big(f^\beta_\alpha\sigma_\beta+f\sigma_\alpha\big)-
      s_3\phi^\beta_{\alpha\beta}
      \Big]\Big[s_2\big(f^\alpha_i\sigma_j+f^\alpha_j\sigma_i\big)+
      s_3\phi^\alpha_{ij}
      \Big]\\&+
      \Big[
      L^\alpha_{\underline{i\alpha}}-s_2\big(f^\alpha_i\sigma_\alpha+f\sigma_i\big)-
      s_3\phi^\alpha_{i\alpha}
      \Big]\Big[
      L^\beta_{\underline{j\beta}}-
      s_2\big(
      f^\beta_j\sigma_\beta+f\sigma_j
      \big)-s_3\phi^\beta_{j\beta}
      \Big],
      \endaligned\label{eq:WbasicfactoredS}\\\displaybreak[0]
      &\aligned
      \mathcal A^i_{jmn}&=-s_2\big(f^i_{[m|n]}\sigma_j-
      f^i_{j|[m}\sigma_{n]}+
      f^i_j\sigma_{[m|n]}+f^i_{[m}\sigma_{j|n]}
      \big)-s_3\phi^i_{j[m|n]}\\&+
      \Big[s_2\big(f^\alpha_j\sigma_m+f^\alpha_m\sigma_j\big)+s_3\phi^\alpha_{jm}\Big]\,
      \Big[s_2\big(f^i_\alpha\sigma_n+f^i_n\sigma_\alpha\big)+s_3\phi^i_{\alpha n}\Big]\\&-
      \Big[s_2\big(f^\alpha_j\sigma_n+f^\alpha_n\sigma_j\big)+s_3\phi^\alpha_{jn}\Big]\,
      \Big[s_2\big(f^i_\alpha\sigma_m+f^i_m\sigma_\alpha\big)+s_3\phi^i_{\alpha m}\Big],
      \endaligned\label{eq:WbasicfactoredA}
    \end{align}

    \noindent is the basic associated invariant for the mapping $\mathcal
    F$.

    \pagebreak

    It is satisfied the next equalities

    \begin{align}
      &\aligned
      0&=\overline R^i_{jmn}-\dfrac1{N+1}\Big[\delta^i_j\overline\rho_{[mn]}
      +\delta^i_{[m}\overline
      L^\alpha_{\underline{j\alpha}\|n]}-
      \delta^i_{[m}\overline\rho_{jn]}\Big]
      -\dfrac1{(N+1)^2}\delta^i_{[m}\widetilde{\overline{\mathcal
      S}}{}_{jn]}+\overline{\mathcal A}{}^i_{jmn}\\&-
      R^i_{jmn}+\dfrac1{N+1}\Big[\delta^i_j\rho_{[mn]}+
      \delta^i_{[m}L^\alpha_{\underline{j\alpha}|n]}-
      \delta^i_{[m}\rho_{jn]}\Big]
      +\dfrac1{(N+1)^2}\delta^i_{[m}\widetilde{\mathcal
      S}{}_{jn]}-\mathcal A^i_{jmn},
      \endaligned
      \label{eq:Wbasicfactored*}
      \\
      &\widetilde{\mathcal S}{}_{ij}=\widetilde{\mathcal
      S}{}_{ji},\mbox{ so }\widetilde{\mathcal S}{}_{[ij]}=0,\label{eq:Wbasicfactored****}\\&
      \delta^i_{m}L^\alpha_{\underline{j\alpha}|n}-
      \delta^i_{m}L^\alpha_{\underline{n\alpha}|j}\overset{(\ref{eq:Wbasicfactored**})}=-\delta^i_mR_{[jn]}
      \Rightarrow\delta^i_{[m}L^\alpha_{\underline{j\alpha}|n]}-
      \delta^i_{[m}L^\alpha_{\underline{n]\alpha}|j}=-\delta^i_{[m}R_{jn]}+\delta^i_{[m}R_{n]j},\label{eq:Laja|n-Lana|j}\\&
      \rho_{[ij]}=s_2
      \big(
      f^\alpha_{[i|j]}\sigma_\alpha-f_{[i}\sigma_{j]}+f^\alpha_{[i}\sigma_{\alpha|j]}+
      f\sigma_{[i|j]}
      \big)+s_3\phi^\alpha_{[i\alpha|j]}
      ,\label{eq:Wbasicfactored*****}
      \\&
      \mathcal A{}^\alpha_{\alpha ij}=
      -s_2\big(
      f^\alpha_{[i|j]}\sigma_\alpha-f_{[i}\sigma_{j]}+f\sigma_{[i|j]}+f^\alpha_{[i}\sigma_{\alpha|j]}
      \big)-s_3\phi^\alpha_{\alpha[i|j]}=-\rho_{[ij]},
      \label{eq:Wbasicfactored******}\\
      &\aligned
      \mathcal A^\alpha_{[ij]\alpha}&=s_2\big(-f_{[i}\sigma_{j]}+
      f^\alpha_{[i|j]}\sigma_\alpha+f^\alpha_{[i}\sigma_{\alpha|j]}+
      f\sigma_{[i|j]}\big)+s_3\phi^\alpha_{[i\alpha|j]}=\rho_{[ij]}.
      \endaligned\label{eq:Wbasicfactored******x}
    \end{align}

    The equation (\ref{eq:Wbasicfactored*}) holds based on the
    invariance $\widetilde{\overline{\mathcal
    W}}{}^i_{jmn}-\widetilde{\mathcal W}{}^i_{jmn}\overset{(\ref{eq:Wbasicfactored})}=0$.  The equations
    (\ref{eq:Wbasicfactored****}--\ref{eq:Wbasicfactored******x}) are
    in effect with respect to the expressions
    (\ref{eq:Wbasicfactoredrho}, \ref{eq:WbasicfactoredS},
    \ref{eq:WbasicfactoredA}) of the corresponding geometrical
    structures.

    If contract the equation (\ref{eq:Wbasicfactored*}) by the
    indices $i$ and $j$, use the symmetry (\ref{eq:Wbasicfactored****}) and the expression (\ref{eq:Wbasicfactored******}),
    and with respect to the invariance $-\overline L{}^\alpha_{[\underline{i\alpha}\|j]}\overset{(\ref{eq:Wbasicfactored**})}=
    \overline R{}_{[ij]}\overset{(\ref{eq:R[ij]inv})}=R_{[ij]}
    \overset{(\ref{eq:Wbasicfactored**})}=-L^\alpha_{[\underline{i\alpha}|j]}$ as well, we will obtain

    \begin{equation*}
      \aligned
      0&=-\dfrac N{N+1}\overline\rho{}_{[mn]}+
      \dfrac1{N+1}\overline\rho{}_{[mn]}\cancelto{-\overline\rho{}_{[mn]}}{+\overline{\mathcal A}{}^\alpha_{\alpha
      mn}}
      +\dfrac N{N+1}\rho_{[mn]}-\dfrac1{N+1}\rho_{[mn]}\cancelto{\rho_{[mn]}}{-
      \mathcal A^\alpha_{\alpha mn}},
      \endaligned
    \end{equation*}

    \noindent i.e.

    \begin{equation}
    \overline\rho{}_{[mn]}=\rho_{[mn]}.
    \label{eq:i=jWbasicfactor*}
    \end{equation}

    Based on the equations (\ref{eq:Wbasicfactored*****},
    \ref{eq:Wbasicfactored******}, \ref{eq:i=jWbasicfactor*}), one gets

    \begin{equation}
    \aligned
    &\overline{\mathcal A}{}^\alpha_{\alpha ij}=\mathcal
    A{}^\alpha_{\alpha ij},\quad\overline{\mathcal
    A}{}^\alpha_{[ij]\alpha}=\mathcal
    A{}^\alpha_{[ij]\alpha}.\endaligned\label{eq:Wrhounutrasnjainv}
    \end{equation}

    Hence, the basic invariant $\widetilde{\mathcal
    W}{}^i_{jmn}$ for the mapping $\mathcal F$ given by the equation (\ref{eq:Wbasicfactored}) reduces to

    \begin{equation}
      \aligned
      \widetilde{\mathcal
      W}{}^i_{jmn}=R^i_{jmn}+\mathcal
      A^i_{jmn}-\dfrac1{N+1}\big(\delta^i_{[m}L^\alpha_{\underline{j\alpha}|n]}-
      \delta^i_{[m}\rho_{jn]}\big)
      -\dfrac1{(N+1)^2}\delta^i_{[m}\widetilde{\mathcal
      S}{}_{jn]}.
      \endaligned\tag{\ref{eq:Wbasicfactored}'}\label{eq:Wbasicfactoredfinal}
    \end{equation}

    The next lemma holds.

    \begin{lem}
    Let $\mathcal F:\mathbb{GA}_N\to\mathbb{G\overline A}{}_N$ be a mapping
    whose deformation tensor is $P^i_{jk}$,
    $P^i_{\underline{jk}}=\overline\omega{}^i_{jk}-\omega^i_{jk}$,
    for the geometrical object $\omega^i_{jk}$ given by the equation
    \emph{(\ref{eq:omegageneral})}.

    The geometrical object $\widetilde{\mathcal W}{}^i_{jmn}$ given
    by the equation \emph{(\ref{eq:Wbasicfactoredfinal})} is the
    associated basic invariant of the Weyl type for the mapping $\mathcal F$.
      \qed
    \end{lem}

  \subsection{Associated derived invariants with respect to basic
  invariants}

  If contract the equality $\widetilde{\overline{\mathcal
  T}}{}^i_{jk}-\widetilde{\mathcal T}{}^i_{jk}=0$ by any pair of
  indices $(i,j)$ or $(i,k)$ as well, we will express the transformation rule for terms in
  brackets from the equation (\ref{eq:ThomasBasicsim2factoredsim}).
  This transformation rule will not change the form of the invariant
  $\widetilde{\mathcal T}{}^i_{jk}$.

  With respect to the invariant $\widetilde{\mathcal W}{}^i_{jmn}$
  for the mapping $\mathcal F$ given by the equation (\ref{eq:Wbasicfactoredfinal}), we get

  \begin{equation}
    \begin{array}{ccc}
      X_{ij}=0,&Y_{ij}=-\dfrac1{(N+1)^2}\Big[(N+1)\big(L^\alpha_{\underline{i\alpha}|j}-\rho_{ij}\big)+\widetilde{\mathcal
      S}{}_{ij}\Big],&Z^i_{jmn}=\mathcal A^i_{jmn}.
    \end{array}\label{eq:XYZ(Wbasic)}
  \end{equation}

  After substituting the terms (\ref{eq:XYZ(Wbasic)}) into the
  equations (\ref{eq:condinv1}, \ref{eq:condinv2}, \ref{eq:condinv4}), and applying the
  invariance $\overline{\mathcal A}{}^\alpha_{[ij]\alpha}=\mathcal A^\alpha_{[ij]\alpha}$, one
  confirms the validity of the next theorem.

  \begin{thm}
    Let $\mathcal F:\mathbb{GA}_N\to\mathbb{G\overline A}{}_N$ be a
    mapping whose deformation tensor is $P^i_{jk}$,
    $P^i_{\underline{jk}}=\overline\omega{}^i_{jk}-\omega^i_{jk}$,
    for the geometrical object $\omega^i_{jk}$ given by the equation
    \emph{(\ref{eq:omegageneral})}.

    The geometrical objects $\widetilde{\mathcal
    W}{}^{(1)i}_{jmn}\equiv\mathcal{\widetilde W}{}^i_{jmn}$,
    $\widetilde{\mathcal
    W}{}^{(2)i}_{jmn}\equiv\widetilde{\mathcal W}{}^{(1)i}_{jmn}$,

    \begin{align}
    &\aligned
      \widetilde{\mathcal
      W}{}^{(4)i}_{jmn}=R^i_{jmn}+\dfrac1{N-1}\delta^i_{[m}R_{\underline{jn}]}+\mathcal
      A{}^i_{jmn}+\dfrac1{N-1}\delta^i_{[m}\mathcal
      A{}^\alpha_{\underline{jn}]\alpha},
      \endaligned\label{eq:FW4}
    \end{align}

    \noindent for the geometrical objects $\rho_{ij}$,
    $\widetilde{\mathcal S}{}_{ij}$, $\mathcal A^i_{jmn}$ given by the equations \emph{(\ref{eq:Wbasicfactoredrho},
    \ref{eq:WbasicfactoredS}, \ref{eq:WbasicfactoredA})}, are invariants for the mapping
    $\mathcal F$.
  \end{thm}

  \begin{cor}
  The geometrical object
    \begin{equation}
      \overset1{\widetilde{\mathcal
      W}}{}^i_{jmn}=R^i_{jmn}-\dfrac1{(N+1)^2}\Big((N+1)
      \big(\delta^i_{[m}L^\alpha_{\underline{j\alpha}|n]}-\delta^i_{[m}\rho_{jn]}\big)+
      \delta^i_{[m}\mathcal
      A^\alpha_{\underline{jn}]\alpha}\Big)+\mathcal A^i_{jmn},
      \label{eq:FW1[1]}
    \end{equation}

    \noindent is an invariant for the mapping $\mathcal F$.

    All associated invariants of the Weyl type for the mapping
    $\mathcal F$ are linear combinations of the invariants
     $\widetilde{\mathcal W}{}^i_{jmn}$, $\widetilde{\mathcal
    W}{}^{(4)i}_{jmn}$, $\overset1{\widetilde{\mathcal
    W}}{}^i_{jmn}$ and the compositions of the invariants of the
    Thomas type for the mapping $\mathcal F$ given by the corresponding of the equations \emph{(\ref{eq:ThomasBasicsim2factoredsim}, \ref{eq:basicthomass1=0})}
    \end{cor}
    \begin{proof}
      With respect to the equations (\ref{eq:propinvW},
      \ref{eq:Wbasicfactoredfinal}), we get

      \begin{equation}
        \begin{array}{ccc}
          X_{ij}=0&Y_{ij}=-\dfrac1{N+1}L^\alpha_{\underline{i\alpha}|j}+
          \dfrac1{N+1}\rho_{ij}-\dfrac1{(N+1)^2}\widetilde{\mathcal
          S}{}_{ij},&Z^i_{jmn}=\mathcal A^i_{jmn}.
        \end{array}\label{eq:XYZ(FW1)}
      \end{equation}

      If substitute the geometrical objects $X_{ij}$, $Y_{ij}$,
      $Z^i_{jmn}$ given by the equation (\ref{eq:XYZ(FW1)}) into the
      equations (\ref{eq:condinv1}, \ref{eq:condinv2},
      \ref{eq:condinv4}), we will obtain the invariants
      $\overset1{\widetilde{\mathcal W}}{}^i_{jmn}$ and
      $\widetilde{\mathcal W}{}^{(4)i}_{jmn}$ given by the equations
      (\ref{eq:FW1[1]}) and (\ref{eq:FW4}) respectively.

      Based on the equations (\ref{eq:propinvW},
      \ref{eq:FW4}), we read

      \begin{equation}
        \begin{array}{ccc}
          X_{ij}=0&Y_{ij}=\dfrac1{N-1}R_{\underline{ij}}+\dfrac1{N-1}
          \mathcal A^\alpha_{\underline{ij}\alpha},&Z^i_{jmn}=\mathcal A^i_{jmn}.
        \end{array}\label{eq:XYZ(FW4)}
      \end{equation}

      From these $X_{ij}$, $Y_{ij}$, $Z^i_{jmn}$ involved into the
      equations (\ref{eq:condinv1}, \ref{eq:condinv2},
      \ref{eq:condinv4}), we get the invariant
      $\widetilde{W}{}^{(4)i}_{jmn}$ for the mapping $\mathcal F$
      given by the equation (\ref{eq:FW4}).

      From the invariant $\overset1{\widetilde{\mathcal
      W}}{}^i_{jmn}$, one gets

      \begin{equation}
        \begin{array}{ccc}
          X_{ij}=0&Y_{ij}=-\dfrac1{N+1}L^\alpha_{\underline{i\alpha}|j}+
          \dfrac1{N+1}\rho_{ij}-\dfrac1{(N+1)^2}\mathcal A^\alpha_{\underline{ij}\alpha},&Z^i_{jmn}=\mathcal A^i_{jmn}.
        \end{array}\label{eq:XYZ(FW1[1])}
      \end{equation}

      After substituting the expressions (\ref{eq:XYZ(FW1[1])}) of
      the geometrical objects $X_{ij}$, $Y_{ij}$, $Z^i_{jmn}$ into
      the equations (\ref{eq:condinv1}, \ref{eq:condinv2},
      \ref{eq:condinv4}), one obtains the invariants
      $\overset1{\widetilde{\mathcal W}}{}^i_{jmn}$ and
      $\widetilde{\mathcal W}{}^{(4)i}_{jmn}$.

      In this way, we completed the recursion for obtaining
      invariants of the Weyl type with respect to the Theorem \ref{theoreminv}. For this reason, any other
      associated invariant of the Weyl type is the
      linear combination of the previously obtained invariants $\widetilde{\mathcal
      W}{}^{(1)i}_{jmn}$, $\widetilde{\mathcal W}{}^{(4)i}_{jmn}$,
      $\overset1{\widetilde{\mathcal W}}{}^i_{jmn}$ and the
      corresponding compositions of the invariants of the Thomas
      type given by the equations
      (\ref{eq:ThomasBasicsim2factoredsim},
      \ref{eq:basicthomass1=0}).
    \end{proof}

    \pagebreak

    \begin{cor}
    If the geometrical object $\mathcal A{}^i_{jmn}$ is

     \begin{equation}
     \mathcal A{}^i_{jmn}=\delta^i_j\mathcal
    P{}_{[mn]}+\delta^i_{[m}\mathcal Q{}_{jn]}+\mathcal
    N^i_{jmn},\label{eq:A(P,Q)}
    \end{equation}

    \noindent then the invariants $\widetilde{\mathcal
    W}{}^{i}_{jmn}$, $\widetilde{\mathcal W}{}^{(4)i}_{jmn}$,
    $\overset1{\mathcal{\widetilde W}}{}^i_{jmn}$ transform to

    \begin{align}
      &\aligned
      \widetilde{\mathcal
      W}{}^{i}_{jmn}&=R^i_{jmn}-\dfrac1{N+1}\big(\delta^i_{[m}L^\alpha_{\underline{j\alpha}|n]}-
      \delta^i_{[m}\rho_{jn]}\big)-\dfrac1{(N+1)^2}\delta^i_{[m}\widetilde{\mathcal
      S}{}_{jn]}\\&+\delta^i_j\mathcal
      P_{[mn]}+\delta^i_{[m}\mathcal Q_{jn]}+\mathcal N^i_{jmn},
      \endaligned\label{eq:FW1(P,Q)}\\\displaybreak[0]
      &\aligned
      \widetilde{\mathcal
      W}{}^{(4)i}_{jmn}&=R^i_{jmn}+\dfrac1{N-1}\delta^i_{[m}R_{\underline{jn}]}+
      \delta^i_j\mathcal P_{[mn]}+\mathcal
      N^i_{jmn}+\dfrac1{N-1}\delta^i_{[m}\mathcal
      N^\alpha_{\underline{jn}]\alpha},
      \endaligned\label{eq:FW4(P,Q)}\\\displaybreak[0]
      &\aligned
      \overset1{\widetilde{\mathcal W}}{}^i_{jmn}&=R^i_{jmn}-
      \dfrac1{N+1}\big(\delta^i_{[m}L^\alpha_{\underline{j\alpha}|n]}-
      \delta^i_{[m}\rho_{jn]}\big)+\delta^i_j\mathcal
      P_{[mn]}+\delta^i_{[m}\mathcal Q_{jn]}+\mathcal N^i_{jmn}\\&+
      \dfrac1{(N+1)^2}\delta^i_{[m}\mathcal
      Q_{\underline{jn}]}-\dfrac1{(N+1)^2(N-1)}\delta^i_{[m}\mathcal
      N^\alpha_{\underline{jn}]\alpha}.
      \endaligned\label{eq:FW1[1](P,Q)}
    \end{align}
  \end{cor}

  \begin{proof}
  If contracts the equality (\ref{eq:A(P,Q)}) by the indices $i$ and
  $n$, one obtains

  \begin{equation}
     \mathcal A{}^\alpha_{\underline{ij}\alpha}=-(N-1)\mathcal Q{}_{\underline{ij}}+\mathcal
    N^\alpha_{\underline{ij}\alpha},\label{eq:A(P,Q)i=nsim}
    \end{equation}

    After substituting the equations (\ref{eq:A(P,Q)},
    \ref{eq:A(P,Q)i=nsim}) into the equations (\ref{eq:Wbasicfactoredfinal},
    \ref{eq:FW4}, \ref{eq:FW1[1]}), one completes the proof for this
    corollary.
  \end{proof}

  \section{Invariants for almost geodesic mappings of third
  type}\label{associatedderivedinvariantssec4}

  In attempt to generalize the concept of geodesics, N. S. Sinyukov
  started the research about almost geodesic mappings
  \cite{sinjukov}. J. Mike\v s and his research group
  have continued the study about almost geodesic mappings of
  symmetric affine connection spaces \cite{mik10,mik5}.

  M. S. Stankovi\'c \cite{ag1,ag2,ag3} generalized the theory
  of almost geodesic mappings of symmetric affine connection spaces.
  Many authors have continued this research
  \cite{mica2,mica1,mica3,ddis3,ddis4} and many others.

  There are three kinds and two types of almost geodesic mappings
  $\mathcal F:\mathbb{GA}_N\to\mathbb{G\overline A}{}_N$.

  We are interested to obtain the invariants for equitorsion almost geodesic
  mappings $\mathcal F:\mathbb{GA}_N\to\mathbb{G\overline A}{}_N$ of
  both first and the second kind.

  The basic equations of the equitorsion almost geodesic mapping
  $\mathcal F:\mathbb{GA}_N\to\mathbb{G\overline A}_N$ of a $p$-th
  kind, $p=1,2$, are

  \begin{equation}
    \left\{\begin{array}{l}
      \overline L{}^i_{\underline{jk}}=L^i_{\underline{jk}}+
      \delta^i_j\psi_k+\delta^i_k\psi_j+\sigma_{jk}\varphi^i,\\
      \varphi^i_{\underset p|j}=\nu_j\varphi^i+\mu\delta^i_j,
    \end{array}\right.\label{eq:pi3pbasicequations}
  \end{equation}

  \noindent for $\varphi^i_{\underset1|j}=\varphi^i_{,j}+L^i_{\alpha
  j}\varphi^\alpha\equiv\varphi^i_{|j}+L^i_{\underset\vee{\alpha j}}\varphi^\alpha$ and
  $\varphi^i_{\underset2|j}=\varphi^i_{,j}+L^i_{j\alpha}\equiv
  \varphi^i_{|j}-L^i_{\underset\vee{\alpha j}}\varphi^\alpha$.

  To complete this research, we need the next equalities

  \begin{align}
    &s_1=1,s_2=0,s_3=1,\phi^i_{jk}=-\dfrac12\sigma_{jk}\varphi^i,\label{eq:pi3s1s2s3phi}\\\displaybreak[0]
    &\rho_{ij}=-\dfrac12\sigma_{i\alpha|j}\varphi^\alpha-\dfrac12\sigma_{i\alpha}\varphi^\alpha_{|j}
    =-\dfrac12\sigma_{i\alpha|j}\varphi^\alpha-
    \dfrac12\sigma_{i\alpha}\nu_j\varphi^\alpha-
    \dfrac12\mu\sigma_{ij}-\dfrac{(-1)^p}2\sigma_{i\alpha}L^\alpha_{\underset\vee{\beta j}}\varphi^\beta
    ,\label{eq:pi3rhoij}\\\displaybreak[0]
    &\aligned\widetilde{\mathcal S}{}_{ij}&=
    -\dfrac{N+1}2\sigma_{ij}\varphi^\alpha\Big[L^\beta_{\underline{\alpha\beta}}+\dfrac12\sigma_{\alpha\beta}\varphi^\beta\Big]+
    \Big[L^\alpha_{\underline{i\alpha}}+\dfrac12\sigma_{i\alpha}\varphi^\alpha\Big]
    \Big[L^\beta_{\underline{j\beta}}+\dfrac12\sigma_{j\beta}\varphi^\beta\Big],
    \endaligned\label{eq:pi3Sij}\\\displaybreak[0]
    &\aligned
    \mathcal
    A^i_{jmn}&=\dfrac12\big(\sigma_{jm|n}\varphi^i+\sigma_{jm}\varphi^i_{|n}\big)+
    \dfrac14\big(\sigma_{jm}\sigma_{\alpha
    n}-\sigma_{jn}\sigma_{\alpha
    m}\big)\varphi^\alpha\varphi^i\\&
    =-\dfrac14\mu\delta^i_{[m}\sigma_{jn]}+\dfrac14\big(\sigma_{jm|n}-\sigma_{jn|m}\big)\varphi^i+
    \dfrac14\big(\sigma_{jm}\sigma_{\alpha
    n}-\sigma_{jn}\sigma_{\alpha
    m}\big)\varphi^\alpha\varphi^i
    \\&+\dfrac14\Big[\sigma_{jm}\big(\nu_n\varphi^i+(-1)^pL^i_{\underset\vee{\alpha
    n}}\varphi^\alpha\big)-\sigma_{jn}\big(\nu_m\varphi^i+(-1)^pL^i_{\underset\vee{\alpha
    m}}\varphi^\alpha\big)\Big].
    \endaligned\label{eq:pi3Aijmn}
  \end{align}

  With respect to the equations (\ref{eq:A(P,Q)},
  \ref{eq:pi3Aijmn}), one gets

  \begin{equation}
    \begin{array}{c}
    \begin{array}{cc}
      \mathcal P_{ij}=0,&\mathcal Q_{ij}=-\dfrac14\mu\sigma_{ij},
      \end{array}\\
      {\aligned
      \mathcal N^i_{jmn}&=\dfrac14\big(\sigma_{jm|n}-\sigma_{jn|m}\big)\varphi^i+
    \dfrac14\big(\sigma_{jm}\sigma_{\alpha
    n}-\sigma_{jn}\sigma_{\alpha
    m}\big)\varphi^\alpha\varphi^i
    \\&+\dfrac14\Big[\sigma_{jm}\big(\nu_n\varphi^i+(-1)^pL^i_{\underset\vee{\alpha
    n}}\varphi^\alpha\big)-\sigma_{jn}\big(\nu_m\varphi^i+(-1)^pL^i_{\underset\vee{\alpha
    m}}\varphi^\alpha\big)\Big].
      \endaligned}
    \end{array}\label{eq:pi3P,Q,N}
  \end{equation}

  The basic invariant for the almost geodesic mapping $\mathcal F$ is

  \begin{footnotesize}
  \begin{equation}
    \aligned
    \widetilde{\mathcal W}{}^i_{jmn}&=R^i_{jmn}-\dfrac{N+3}{4(N+1)}\delta^i_{[m}\mu\sigma_{jn]}+\dfrac14\big(\sigma_{jm|n}-\sigma_{jn|m}\big)\varphi^i+
    \dfrac14\big(\sigma_{jm}\sigma_{\alpha
    n}-\sigma_{jn}\sigma_{\alpha
    m}\big)\varphi^\alpha\varphi^i
    \\&+\dfrac14\Big[\sigma_{jm}\big(\nu_n\varphi^i+(-1)^pL^i_{\underset\vee{\alpha
    n}}\varphi^\alpha\big)-\sigma_{jn}\big(\nu_m\varphi^i+(-1)^pL^i_{\underset\vee{\alpha
    m}}\varphi^\alpha\big)\Big]\\&-
    \dfrac1{N+1}\Big[\delta^i_{[m}L^\alpha_{\underline{j\alpha}|n]}+
    \dfrac12\big(\delta^i_{[m}\sigma_{j\alpha|n]}+\delta^i_{[m}\sigma_{j\alpha}\nu_{n]}\big)\varphi^\alpha+\dfrac{(-1)^p}2\delta^i_{[m}\sigma_{j\alpha}
    L^\alpha_{\underset\vee{\beta
    n}]}\varphi^\beta\Big]\\&+\dfrac1{2(N+1)}\delta^i_{[m}\sigma_{jn]}\varphi^\alpha
    \Big[L^\beta_{\underline{\alpha\beta}}+\dfrac12\sigma_{\alpha\beta}\varphi^\beta\Big]-
    \dfrac1{(N+1)^2}\Big[L^\alpha_{\underline{j\alpha}}+\dfrac12\sigma_{j\alpha}\varphi^\alpha\Big]
    \Big[\delta^i_{[m}L^\beta_{\underline{n]\beta}}+\dfrac12\delta^i_{[m}\sigma_{n]\beta}\varphi^\beta\Big].
    \endaligned\label{eq:pi3basic}
  \end{equation}
  \end{footnotesize}

  The the fourth associated derived invariant for the mapping
  $\mathcal F$ is

  \begin{footnotesize}
    \begin{align}
      &\aligned
      \widetilde{\mathcal
      W}{}^{(4)i}_{jmn}&=R^i_{jmn}+\dfrac1{N-1}\delta^i_{[m}R_{jn]}+
      \dfrac14\big(\sigma_{jm|n}-\sigma_{jn|m}\big)\varphi^i+
    \dfrac14\big(\sigma_{jm}\sigma_{\alpha
    n}-\sigma_{jn}\sigma_{\alpha
    m}\big)\varphi^\alpha\varphi^i
    \\&+\dfrac14\Big[\sigma_{jm}\big(\nu_n\varphi^i+(-1)^pL^i_{\underset\vee{\alpha
    n}}\varphi^\alpha\big)-\sigma_{jn}\big(\nu_m\varphi^i+(-1)^pL^i_{\underset\vee{\alpha
    m}}\varphi^\alpha\big)\Big]
    \\&+\dfrac1{4(N-1)}\big(\delta^i_{[m}\sigma_{jn]|\alpha}-\delta^i_{[m}\sigma_{j\alpha|n]}\big)\varphi^\alpha+
    \dfrac1{4(N-1)}\delta^i_{[m}\sigma_{jn]}\big(\sigma_{\alpha\beta}+\nu_\alpha\varphi^\alpha+
    (-1)^pL^\beta_{\underset\vee{\alpha\beta}}\varphi^\alpha\big)\\&-
    \dfrac1{4(N-1)}\big(\delta^i_{[m}\sigma_{j\alpha}\sigma_{n]\beta}\varphi^\alpha\varphi^\beta+
    \delta^i_{[m}\sigma_{j\alpha}\nu_{n]}\varphi^\alpha+(-1)^p\delta^i_{[m}\sigma_{j\beta}L^\beta_{\underset\vee{\alpha
    n}]}\varphi^\alpha\big).
      \endaligned\label{eq:pi3condinv4}
    \end{align}
  \end{footnotesize}

  The invariant $\overset1{\widetilde{\mathcal W}}{}^i_{jmn}$ for
  the almost geodesic mapping $\mathcal F$ is

  \begin{footnotesize}
    \begin{equation}
      \aligned
      \overset1{\widetilde{\mathcal
      W}}{}^i_{jmn}&=R^i_{jmn}-\dfrac1{2(N+1)}\Big[2\delta^i_{[m}L^\alpha_{\underline{j\alpha}|n]}+\big(\delta^i_{[m}
      \sigma_{j\alpha|n]}+\delta^i_{[m}\sigma_{j\alpha}\nu_{n]}\big)\varphi^\alpha+
      (-1)^p\delta^i_{[m}\sigma_{j\alpha}L^\alpha_{\underset\vee{\beta
      n}]}\varphi^\beta\Big]\\&-
      \dfrac{(N+2)^2}{4(N+1)^2}\delta^i_{[m}\mu\sigma_{jn]}+\dfrac14\big(\sigma_{jm|n}-\sigma_{jn|m}\big)\varphi^i+
    \dfrac14\big(\sigma_{jm}\sigma_{\alpha
    n}-\sigma_{jn}\sigma_{\alpha
    m}\big)\varphi^\alpha\varphi^i
    \\&+\dfrac14\Big[\sigma_{jm}\big(\nu_n\varphi^i+(-1)^pL^i_{\underset\vee{\alpha
    n}}\varphi^\alpha\big)-\sigma_{jn}\big(\nu_m\varphi^i+(-1)^pL^i_{\underset\vee{\alpha
    m}}\varphi^\alpha\big)\Big]\\&-\dfrac1{4(N+1)^2(N-1)}\Big[
    \big(\delta^i_{[m}\sigma_{jn]|\alpha}-\delta^i_{[m}\sigma_{j\alpha|n]}\big)\varphi^\alpha+
    \big(\delta^i_{[m}\sigma_{jn]}\sigma_{\alpha\beta}-\delta^i_{[m}\sigma_{j\beta}\sigma_{n]\alpha}\big)\varphi^\alpha\varphi^\beta\Big]\\&
    -\dfrac1{4(N+1)^2(N-1)}\Big[\delta^i_{[m}\sigma_{jn]}\big(\nu_\alpha\varphi^\alpha+(-1)^pL^\beta_{\underset\vee{\alpha\beta}}\varphi^\alpha\big)
    -\sigma_{j\beta}\big(\delta^i_{[m}\nu_{n]}\varphi^\beta+(-1)^p\delta^i_{[m}L^\beta_{\underset\vee{\alpha
    n}]}\varphi^\alpha\big)\Big].
      \endaligned
    \end{equation}
  \end{footnotesize}

  \section{Conclusion}

  We developed the
  methodology for obtaining associated invariants for mappings
  between non-symmetric affine connection spaces in this paper.

  In the Section \ref{associatedderivedinvariantssec2}, it was continued
  the research from  about derived invariants for
  geometrical mappings \big(see \cite{jageninv}\big). We obtained three invariants
  $W^{(1)i}_{jmn}$, $W^{(2)i}_{jmn}$, $W^{(4)i}_{jmn}$ from the
  invariant $W^i_{jmn}$ \Big(see the Theorem \ref{theoreminv}, Eqs. (\ref{eq:propinvW},
  \ref{eq:condinv1}, \ref{eq:condinv2}, \ref{eq:condinv4}\big)\Big). In
  this section, we founded the auxiliary invariant for any mapping
  which will simplify some of the invariants obtained until now.

  In the Section \ref{associatedderivedinvariantssec3}, we obtained
  the associated basic and associated derived invariants for a mapping
  $\mathcal F:\mathbb{GA}_N\to\mathbb{G\overline A}{}_N$.

  In the Section \ref{associatedderivedinvariantssec4}, it were applied
  the results obtained in the Section
  \ref{associatedderivedinvariantssec3} for finding invariants for
  almost geodesic mappings of the third type.

  In the future researches, we will discuss about the space of invariants
  with respect to the results obtained in the Section
  \ref{associatedderivedinvariantssec2}. The
  results obtained in this paper will be generalized with respect to the transformation
  rules of torsion tensors under different mappings.

  \section*{Acknowledgements}

  This paper is financially supported by Serbian Ministry of
  Education, Science and Technological Development, Grant. No.
  174012.

  \end{document}